\RequirePackage{fix-cm}
\documentclass[smallextended]{svjour3}       % onecolumn (second format)
\smartqed  % flush right qed marks, e.g. at end of proof
\usepackage{graphicx}
\usepackage{times}
\usepackage[T1]{fontenc}
\usepackage{mathrsfs}
\usepackage{latexsym}
\usepackage{amsmath,amsfonts}
\input amssym.def
\input amssym.tex
\usepackage{hyperref}
\usepackage{color}
\usepackage{breakurl}
\usepackage{comment}
\newcommand{\bburl}[1]{\textcolor{blue}{\url{#1}}}

\newcommand\be{\begin{equation}}
\newcommand\ee{\end{equation}}
\newcommand\bea{\begin{eqnarray}}
\newcommand\eea{\end{eqnarray}}
\newcommand\bi{\begin{itemize}}
\newcommand\ei{\end{itemize}}
\newcommand\ben{\begin{enumerate}}
\newcommand\een{\end{enumerate}}

\numberwithin{equation}{section}

\newtheorem{thm}{Theorem}[section]

\newtheorem{lem}[thm]{Lemma}

\newtheorem{exa}[thm]{Example}

\newtheorem{rek}[thm]{Remark}

\numberwithin{thm}{section}

\newcommand{\twocase}[5]{#1 \begin{cases} #2 & \text{#3}\\ #4
&\text{#5} \end{cases}   }

\newcommand{\R}{\ensuremath{\mathbb{R}}}

\newcommand{\Z}{\ensuremath{\mathbb{Z}}}
\newcommand{\Q}{\mathbb{Q}}
\newcommand{\N}{\mathbb{N}}

\begin{document}

\title{Recurrence Relations and Benford's Law\thanks{This work was supported in part by NSF Grants DMS1561945 and DMS1659037, the Finnerty Fund, the University of Michigan, and Williams College. }
}
%%%%%\subtitle{Recurrence Relations and Benford's Law}


\author{Madeleine Farris, Noah Luntzlara, Steven J. Miller, Lily Shao and Mengxi Wang}

\authorrunning{Farris, Luntzlara, Miller, Shao and Wang} % if too long for running head

\institute{Madeleine Farris, Department of Mathematics, Wellesley College,  \email{mfarris@wellesley.edu}  \\
Noah Luntzlara, Department of Mathematics, University of Michigan, \email{nluntzla@umich.edu} \\
Steven J. Miller, Department of Mathematics and Statistics, Williams College, \email{sjm1@williams.edu} \\
Lily Shao, Department of Mathematics and Statistics, Williams College, \email{ls12@williams.edu} \\
Mengxi Wang, Department of Mathematics, University of Michigan, \email{mengxiw@umich.edu} \\
}

\date{Received: date / Accepted: date}

\maketitle

\begin{abstract} There are now many theoretical explanations for why Benford's law of digit bias surfaces in so many diverse fields and data sets. After briefly reviewing some of these, we discuss in depth recurrence relations. As these are discrete analogues of differential equations and model a variety of real world phenomena, they provide an important source of systems to test for Benfordness. Previous work showed that fixed depth recurrences with constant coefficients are Benford modulo some technical assumptions which are usually met; we briefly review that theory and then prove some new results extending to the case of linear recurrence relations with non-constant coefficients. We prove that, for certain families of functions $f$ and $g$, a sequence generated by a recurrence relation of the form $a_{n+1} = f(n)a_n + g(n)a_{n-1}$ is Benford for all initial values.  The proof proceeds by parameterizing the coefficients to obtain a recurrence relation of lower degree, and then converting to a new parameter space. From there we show that for suitable choices of $f$ and $g$ where $f(n)$ is nondecreasing and $g(n)/f(n)^2 \to 0$ as $n \to \infty$, the main term dominates and the behavior is equivalent to equidistribution problems previously studied. We also describe the results of generalizing further to higher-degree recurrence relations and multiplicative recurrence relations with non-constant coefficients, as well as the important case when $f$ and $g$ are values of random variables.

\end{abstract}

\tableofcontents

%%%%%%%%%%%%%%%%%%%%%%%%%%%%%%%%%%%%%%%%%%%%%%%%%%%%%%%%%%%%%%%%%%%%%%%%%%%%%%%%%%%%%%%%%%%%%%%%%%%%%%%%%%%%%%%%%%%%%%%%%%%%%%%%%%%%%%%%%%%%%%%%%%%%%%%%%%
%%%%%%%%%%%%%%%%%%%%%%%%%%%%%%%%%%%%%%%%%%%%%%%%%%%%%%%%%%%%%%%%%%%%%%%%%%%%%%%%%%%%%%%%%%%%%%%%%%%%%%%%%%%%%%%%%%%%%%%%%%%%%%%%%%%%%%%%%%%%%%%%%%%%%%%%%%
%%%%%%%%%%%%%%%%%%%%%%%%%%%%%%%%%%%%%%%%%%%%%%%%%%%%%%%%%%%%%%%%%%%%%%%%%%%%%%%%%%%%%%%%%%%%%%%%%%%%%%%%%%%%%%%%%%%%%%%%%%%%%%%%%%%%%%%%%%%%%%%%%%%%%%%%%%
\section{Introduction}

%%%%%%%%%%%%%%%%%%%%%%%%%%%%%%%%%%%%%%%%%%%%%%%%%%%%%%
%%%%%%%%%%%%%%%%%%%%%%%%%%%%%%%%%%%%%%%%%%%%%%%%%%%%%%
\subsection{History}

In 1938 Frank Benford \cite{Ben} observed that in many numerical datasets, the leading digit is not equidistributed among $\{1,\dots,9\}$ as one might expect, but instead heavily biased towards low digits, particularly 1. Frequently the probability of a number having first digit $d$ base-$b$ is $\log_{b}\left(1+1/d\right)$ (so base-$10$ it ranges from about 30\% for a first digit of 1, down to about 4.6\% for a leading digit of 9); this phenomenon became known as Benford's Law. See [BerH1, BerH2, Hi, Mi] and the references therein for some of the history, theory and applications. In addition to being of theoretical interest, Benford's Law has found applications in numerous fields from data integrity (used to detect tax, voter and data fraud) to computer science (designing optimal systems to minimize rounding errors); many of these diverse systems are discussed in detail in the edited book \cite{Mi}.   %%%%\cite{BH1, BH2, Hi, Mi}

To give just a few examples, in \cite{BBH,KM} it was proved that many dynamical systems exhibit Benford behavior, including most power, exponential and rational functions, linearly-dominated systems, non-autonomous dynamical systems, the Riemann zeta function, the $3x + 1$ Problem, and more. Depending on the structure of the system, different techniques are better suited for the analysis. Below we assume our numbers are positive and work in base 10; one can easily generalize to other bases, and if we have complex numbers we can look at their absolute value (though we must exclude zeros). Most of these methods start with the following observation; note $y$ modulo $1$ (or $y \bmod 1$) means the fractional part of $y$.

\begin{lem} \label{lem:benfordEquidistr}
A sequence $\{a_n\}$ is Benford if and only if the sequence $\{\log_{10}a_n\}$ is equidistributed modulo $1$.
\end{lem}

To see this, given any $x$ we can write it in scientific notation as $x = S_{10}(x) 10^{k(x)}$, where $S_{10}(x) \in [1,10)$ is the significand and $k(x)$ is an integer. Then $\log_{10} x \bmod 1 = \log S_{10}(x) \in [0,1)$, two numbers $x$ and $\widetilde{x}$ have the same leading digits if and only if they have the same significand, and the logarithms modulo one being equidistributed means that for a sequence $\{x_n\}$ with $y_n = \log_{10} x_n \bmod 1$ for any $[a,b) \subset [0,1]$ that \be \lim_{N\to\infty} \frac{\#\{n \le N:\ y_n \in [a,b]\}}{N} \ = \ b-a. \ee The equivalence of this equidistribution and Benford's law is immediate. Taking $[a,b) = [\log_{10}(d), \log_{10}(d+1))$ gives $b-a = \log_{10}(1 + 1/d)$, and the $y_n \in [a,b)$ are just those where the first digit of $x_n$ is $d$ (from exponentiating).

Often techniques from Fourier Analysis are very useful in proving Benford behavior; this is because we want to study logarithms modulo 1, and the exponential function \be \exp(2\pi i\theta)\ =\ \cos(2\pi\theta) + i \sin(2\pi\theta)\ee is ideally suited to such problems as we can drop the integer part of the argument without changing the value: \be \exp\left(2\pi i \log_{10} x_n\right) \ = \ \exp\left(2\pi i (\log_{10} x_n \bmod 1)\right). \ee Another common approach is to apply the Central Limit Theorem. For example, if we have a process that is a product of independent random variables, by taking logarithms we have a sum of related independent random variables. Frequently the Central Limit Theorem kicks in, and the resulting sum converges to a Gaussian whose variance diverges to infinity. If we look at these Gaussians modulo 1, they converge to the uniform distribution on $[0,1]$, and hence we again find Benford behavior; a good way to prove the convergence to the uniform is to apply Poisson Summation.

Finally, we note that instead of looking at just the first digit one can look at the distribution of the significand. A system is said to be strongly Benford if the probability of a significand of at most $s$ is $\log_{10} s$. Frequently such systems are called Benford and not strongly Benford; we follow that convention here. We end this subsection by recording a useful observation.

\begin{lem} \label{lem:benfordAsym}
If a sequence $\{a_n\}$ is Benford and $\lim_{n\to\infty} \left(b_n-a_n\right)=0$, then $\{b_n\}$ is Benford as well.
\end{lem}

The above lemma is false if the sequence is not strong Benford, because a tiny perturbation can influence the behavior of the leading digit of a Benford sequence. Our goal below is to highlight the main ideas behind one of the most common methods of proving Benford behavior, Weyl's Theorem, and apply it to recurrence relations.

%%%%%%%%%%%%%%%%%%%%%%%%%%%%%%%%%%%%%%%%%%%%%%%%%%%%%%
%%%%%%%%%%%%%%%%%%%%%%%%%%%%%%%%%%%%%%%%%%%%%%%%%%%%%%
\subsection{Results}

In this paper we concentrate on recurrence relations for several reasons. As they are discrete analogues of differential equations, they model many natural phenomena. Further, the proof for the case of linear recurrences of fixed depth and constant coefficients, which are very important cases, are easily analyzed. These have long been known to obey Benford's Law (see for example \cite{MT,NS}), and have applications ranging from the Fibonacci numbers to the stock market to analyzing gambling strategies to population dynamics. After briefly reviewing these proofs, we extend these results to new families of linear recurrences with non-constant coefficients and non-linear recurrences.

To motivate our question, we quickly review a representative example from mathematical biology. Consider a population where for simplicity there are only four groups: those just born, and those that are 1, 2 or 3 years old. Assume each pair that is one year old gives birth to two new pairs, and each pair that is two years old gives birth to one pair. If we let $a(n)$ denote the number of pairs of newborns at time $n$, $b(n)$ the number of pairs of one year olds at time $n$, and so on, we have the following relation:
\be \left(\begin{array}{c}
                        a(n+1)    \\
                        b(n+1)\\
                        c(n+1) \\
                        d(n+1)
                          \end{array}\right)  \ = \
\left(\begin{array}{cccc}
                        0  & 2 & 1 & 0 \\
                        1  & 0 & 0 & 0 \\
                        0  & 1 & 0 & 0 \\
                        0  & 0 & 1 & 0
                          \end{array}\right)
\left(\begin{array}{c}
                        a(n)    \\
                        b(n)\\
                        c(n) \\
                        d(n)
                          \end{array}\right). \ee
While the above model has the advantage of being mathematically tractable and we can write down a closed form expression for the population at time $n$, it suffers from unrealistic assumptions that the birth rate is constant every year, and that each member of the community never dies until year four, when they all die together. A more accurate model would replace the constants with random variables; here in the first row we might have variables with means respectively 2 and 1, while in the other rows they would probably be random variables with means a little below 1 (to account for natural deaths or predation):
\be \left(\begin{array}{c}
                        a(n+1)    \\
                        b(n+1)\\
                        c(n+1) \\
                        d(n+1)
                          \end{array}\right)  \ = \
\left(\begin{array}{cccc}
                        0  & \beta_1(n) & \beta_2(n) & 0 \\
                        \gamma_1(n)  & 0 & 0 & 0 \\
                        0  & \gamma_2(n) & 0 & 0 \\
                        0  & 0 & \gamma_3(n) & 0
                          \end{array}\right)
\left(\begin{array}{c}
                        a(n)    \\
                        b(n)\\
                        c(n) \\
                        d(n)
                          \end{array}\right). \ee

It is our desire to understand problems such as the above that motivated this work. We being by considering a simpler case, families of sequences generated by recurrence relations of the form
\begin{equation}
a_{n+1}\ =\ f(n)a_n + g(n)a_{n-1}.
\end{equation}
We are not able use any of the standard methods, such as characteristic polynomials, which work for linear recurrences, but we still want a closed form for the sequence. To this end, we introduce auxiliary functions $\lambda(n), \mu(n)$ satisfying
\begin{equation}
a_{n+1} - \lambda(n)a_n \ = \ \mu(n)(a_n - \lambda(n-1)a_{n-1}).
\end{equation}
These auxiliary functions make it possible to effectively reduce the degree of the recurrences when we consider the related sequence $\{a_n - \lambda(n-1)a_{n-1}\}$. This results in the closed form \begin{equation} a_{n+1} \ = \ \left(a_2 - \lambda(1)a_1 \right) \left(\sum_{k=2}^{n} \prod_{i=k+1}^{n}\lambda(i) \prod_{j=2}^{k}\mu(j)\right) + a_2\prod_{i=2}^{n}\lambda(i). \end{equation} Although this formula is not reasonable to work with directly, under certain conditions on $f$ and $g$ it splits into an error term and a main term. The error term converges to zero in the limit, and the main term is simple enough to work with, letting us analyze the Benfordness of the sequence $\{a_n\}$. Our main result is the following.

\begin{thm}\label{thm:main} Let $a_{n+1} = f(n)a_n + g(n)a_{n-1}$. Suppose the functions $f(n)$ and $g(n)$ satisfy $f(n)$ is non-decreasing and $$\lim_{n \to \infty} \frac{g(n)}{f(n)^2} \ =\ 0.$$ Then $\{a_n\}$ is Benford if and only if $\left(\prod_{i=1}^{n} \mu(i)\right)$ is Benford, where $\mu(n)$ is the auxiliary function described above.
\end{thm}

Section \ref{section:examples} gives some examples of recurrent sequences which our results show are Benford, including cases when $f(n)$ and $g(n)$ are random variables. We give two representative examples here. The coefficients for the recurrence relations are deterministic functions in the first and random variables in the second.

\begin{exa} \label{ex:introFctn}
If $\mu(k) = \exp(\alpha h(k))$ where $\alpha$ is irrational and $h(k)$ is a monic polynomial, then $\{a_n\}_{n=1}^{\infty}$ follows Benford's law.
\end{exa}

\begin{exa} \label{ex:introRV}
Suppose $\mu(n) \sim h(n) U_n$, where $U_n$ are independent random variables uniformly distributed on $[0, 1]$ and $h(n)$ is a deterministic function in $n$ such that $\prod_{i=1}^{n} h(i)$ is Benford. Then $\{a_n\}_{n=1}^{\infty}$ follows Benford's Law.
\end{exa}

In Section \ref{section:higherDegree} we show that this result can be generalized to higher-degree recurrences $a_{n+1} = f_1(n)a_n + f_2(n)a_{n-1} + \cdots + f_{k}(n)a_{n-k+1}$. %Parameterizing by introducing similar auxiliary functions, we get the conclusion that under certain conditions the sequence $\{a_n\}$ is Benford if and only if a multiplicative main term is Benford.
In Section \ref{section:multi} we formulate analogous results on Benford behavior of sequences generated by multiplicative recurrence relations \be A_{n+1} \ =\ A_n^{f_1(n)} A_{n-1}^{f_2(n)} \cdots A_{n-k+1}^{f_k(n)}.\ee Using the closed form of the sequence generated by its corresponding linear recurrence %(i.e., $a_{n+1} = f_1(n)a_n + f_2(n)a_{n-1} + \cdots + f_{k}(n)a_{n-k+1}$),
we find that the sequence $\{A_n\}$ is Benford if and only if the main term of $\{a_n\}$ is equidistributed modulo 1.

\begin{acknowledgements}
Much of the analysis was done during the 2018 Williams College SMALL REU Program, and we are grateful to our colleagues there, as well as to participants of the Conference on Benford's Law Application in Stresa, Italy, for helpful comments.
\end{acknowledgements}


%%%%%%%%%%%%%%%%%%%%%%%%%%%%%%%%%%%%%%%%%%%%%%%%%%%%%%%%%%%%%%%%%%%%%%%%%%%%%%%%%%%%%%%%%%%%%%%%%%%%%%%%%%%%%%%%%%%%%%%%%%%%%%%%%%%%%%%%%%%%%%%%%%%%%%%%%%
%%%%%%%%%%%%%%%%%%%%%%%%%%%%%%%%%%%%%%%%%%%%%%%%%%%%%%%%%%%%%%%%%%%%%%%%%%%%%%%%%%%%%%%%%%%%%%%%%%%%%%%%%%%%%%%%%%%%%%%%%%%%%%%%%%%%%%%%%%%%%%%%%%%%%%%%%%
%%%%%%%%%%%%%%%%%%%%%%%%%%%%%%%%%%%%%%%%%%%%%%%%%%%%%%%%%%%%%%%%%%%%%%%%%%%%%%%%%%%%%%%%%%%%%%%%%%%%%%%%%%%%%%%%%%%%%%%%%%%%%%%%%%%%%%%%%%%%%%%%%%%%%%%%%%
\section{Fixed Depth Constant Coefficient Linear Recurrences}

We briefly review the theory of fixed depth constant coefficient linear recurrences (see \cite{MT} for complete details); these are relations of the form \be\label{eq:generalfixeddepthconstlinear} a_{n+1} \ = \ c_1 a_n + \cdots + c_L a_{n+1-L}, \ee where $c_1, \dots, c_n$ are fixed complex numbers and $L$ is a positive integer. We first quickly derive a tractable closed form expression for the solutions, and then show that for most recurrences and most initial conditions, one has Benford behavior.

It has long been known that almost all sequences defined by linear recurrences with constant coefficients and fixed depth obey Benford's law. The main ingredient in these proofs is Weyl's equidistribution theorem (see for example \cite{MT}).

\begin{thm}[Equidistribution Theorem]\label{thm:weylequidistributionkis1} If $\alpha$ is irrational, then the fractional parts of $n \alpha \bmod 1$ are equidistributed. \end{thm}

In addition to being sufficient, this condition is also necessary; if $\alpha$ is rational, say $\alpha = p/q$, then $n\alpha \bmod 1$ only takes on finitely many values (in this case, at most $q$).

For example, if $a_{n+1} = 2^n$ and $a_1 = 1$ then $a_n = 2^n$. To see if it is Benford, we compute \be y_n \ = \ \log_{10}(2^n) \bmod 1 \ = \ n \log_{10} 2 \bmod 1; \ee thus $\{2^n\}$ is Benford base 10 as $\log_{10} 2$ is irrational.\footnote{If $\log_{10} 2$ equals $p/q$ then $2 = 10^{p/q}$, so $2^{q-p} = 5^p$ and thus $p = q-p = 0$, which is impossible.} Not surprisingly, if instead we had $a_{n+1} = 100 a_n$ then $a_n = 100^n = 10^{2n}$ which is clearly not Benford, as each number has first digit 1; note $\log_{10}(100) = 2$, which is rational and not irrational.

As one of our goals is to highlight how to prove Benford behavior, we sketch the proof of Weyl's Equidistribution Theorem in Appendix \ref{sec:proofweyl}.

%%%%%%%%%%%%%%%%%%%%%%%%%%%%%%%%%%%%%%%%%%%%%%%%%%%%%%%%%%%%%
%%%%%%%%%%%%%%%%%%%%%%%%%%%%%%%%%%%%%%%%%%%%%%%%%%%%%%%%%%%%%
%%%%%%%%%%%%%%%%%%%%%%%%%%%%%%%%%%%%%%%%%%%%%%%%%%%%%%%%%%%%%
\subsection{Generalized Binet's Formula}

The simplest case of \eqref{eq:generalfixeddepthconstlinear} is when $L=1$, in which case \be a_{n+1} \ = \ c a_n, \ee which has the solution $a_n = c a_1^n$.

Depth one constant coefficient linear recurrences are trivially solved, as we have $a_n = c r^n$ for some constants $c$ and $r$, and $\{a_n\}$ will be Benford if and only if $\log_{10} r$ is irrational. For the general case as in \eqref{eq:generalfixeddepthconstlinear}, we have a similar relation. The most famous depth two relation is the Fibonacci, where $F_{n+1} = F_n + F_{n-1}$ and $F_1 = F_2 = 2$; in this case we find \be\label{eq:binet} F_{n+1} \ = \ \frac1{\sqrt{5}}\left(\frac{1+\sqrt{5}}{2}\right)^n -  \frac1{\sqrt{5}}\left(\frac{1-\sqrt{5}}{2}\right)^n. \ee We briefly sketch the proof, and then discuss how to generalize to other recurrences.

As $F_n \ge F_{n-1}$ we have \be 2 F_{n-1} \ \le \ F_{n+1} \ \le \ 2 F_n. \ee Thus the first inequality tells us every time $n$ increases by 2 our number at least doubles, so $F_n \ge \sqrt{2}^n$, while the second inequality tells us every time $n$ increases by 1 our number at most doubles, so $F_n \le 2^n$. As $F_n$ is sandwiched between two exponentially growing functions, it is reasonable to guess that $F_n$ equals $r^n$. Substituting that into the recurrence, we get \be r^{n+1} \ = \ r^n + r^{n-1}, \ee which leads to the \emph{characteristic polynomial} \be r^2 - r - 1 \ = \ 0, \ee which has solutions \be r_1 \ := \ \frac{1+\sqrt{5}}{2}\ \ \ {\rm and} \ \ \ r_2 \ := \ \frac{1-\sqrt{5}}{2}. \ee As we have a linear relation, note that any linear combination of solutions is a solution, and thus the most general solution to the Fibonacci recurrence is \be F_n \ = \ \gamma_1 r_1^n + \gamma_2 r_2^n. \ee To determine $c_1$ and $c_2$ we just use the initial conditions that $F_1 = F_2 = 1$ (or $F_0 = 0$ and $F_1 = 1$). After some straightforward algebra, we reach \eqref{eq:binet}, which is known as \emph{Binet's Formula}.

A similar formula holds for the more general recurrence in \eqref{eq:generalfixeddepthconstlinear}. We again try $a_n = r^n$, and obtain the characteristic polynomial \be r^L - c_1 r^{L-1} - \cdots - c_2 r - c_1 \ = \ 0. \ee If this polynomial has $L$ distinct roots, then given any set of $L$ initial conditions there are $L$ constants $\gamma_1, \dots, \gamma_L$ such that \be a_n \ = \ \gamma_1 r_1^n + \cdots + \gamma_L r_L^n. \ee If there are repeated roots the above must be modified slightly; while we will only discuss distinct roots in the next subsection, for completeness we state the general case. If the roots are $r_1, \dots, r_k$ with multiplicities $m_1, \dots, m_k$ then the general solution is \bea a_n & \ = \ & \left(\gamma_{1,1} n^{m_1-1} + \gamma_{1,2} n^{m_1-2} + \cdots + \gamma_{1,m_1}\right) r_1^n + \cdots \nonumber\\ & &  \ \ \ \ \ \ \ \ \ \ \ \cdots\ +  \left(\gamma_{k,1} n^{m_k-1} + \gamma_{k,2} n^{m_k-2} + \cdots + \gamma_{k,m_k}\right) r_k^n. \eea

We quickly motivate this answer by considering a depth two relation with the two roots equal. We slightly perturb the recurrence (by some parameter $\epsilon$) so the two roots are different, and we can write the general solution as \bea a_n(\epsilon) & \ = \ & \gamma_1(\epsilon) r_1(\epsilon)^n + \gamma_2 r_2(\epsilon)^n \nonumber\\ & \ = \ &  \widetilde{\gamma_1}(\epsilon)\frac{r_1(\epsilon)^n - r_2(\epsilon)^n}{r_1(\epsilon) - r_2(\epsilon)} + \widetilde{\gamma_2}(\epsilon)\frac{r_1(\epsilon)^n + r_2(\epsilon)^n}{r_1(\epsilon) + r_2(\epsilon)}. \eea If we take the limit as the perturbation tends to zero, the first term converges to $n r^{n-1}$ while the second converges to $r^n$, which highlights where the polynomials arise.

%%%%%%%%%%%%%%%%%%%%%%%%%%%%%%%%%%%%%%%%%%%%%%%%%%%%%%%%%%%%%
%%%%%%%%%%%%%%%%%%%%%%%%%%%%%%%%%%%%%%%%%%%%%%%%%%%%%%%%%%%%%
%%%%%%%%%%%%%%%%%%%%%%%%%%%%%%%%%%%%%%%%%%%%%%%%%%%%%%%%%%%%%
\subsection{Benford Behavior}

We follow the presentation in \cite{MT}. We concentrate on the simpler case with distinct roots, though with slightly more effort one could handle the case of repeated roots. Note that almost surely a random polynomial has distinct roots (and similarly for the other conditions we assume below). We do need to assume the largest root is not of absolute value 1, as if that happened we could have all the terms of approximately the same magnitude.

\begin{thm}\label{thmrecisben} Let $a_n$ satisfy the recurrence relation \eqref{eq:generalfixeddepthconstlinear} and assume there are $L$
distinct roots. Assume $|r_1| \neq 1$ with $|r_1|$ the
largest absolute value of the roots. Further, assume the initial
conditions are such that the coefficient of $r_1$ is non-zero in the Generalized Binet Formula expansion of $a_n$. If
$\log_{10} |r_1| \not\in \Q$, then $a_n$ is Benford.
\end{thm}

\begin{proof} By the generalized Binet formula we have for any set of initial conditions that there exist constants $\gamma_1, \dots, \gamma_L$ such that \be a_n \ = \ \gamma_1 r_1^n + \cdots + \gamma_L r_L^n, \ee where the $r_i$ are the roots of the characteristic polynomial. By assumption, $\gamma_1 \neq 0$.
For simplicity we assume $r_1 > 0$, $r_1 > |r_2|$ and $\gamma_1 > 0$. Set $y_n = \log_{10} a_n$. It suffices to show $y_n$ is equidistributed modulo
$1$ to prove that $a_n$ is Benford. We have\footnote{We are using big-Oh notation: $f(x) = O(g(x))$ at infinity if there exists an $x_0$ and a $C>0$ such that for all $x \ge x_0$ we have $|f(x)| \le C g(x)$; it is big-Oh at zero if instead it is true for all $x \le x_0$.} \bea\label{eq:mainbigoh} a_n & \ = \ &  \gamma_1 r_1^n \left[1 + O\left(
\frac{L \gamma r_2^n}{r_1^n} \right) \right], \eea where $\gamma = \max_i
|\gamma_i| + 1$ (so $L\gamma > 1$ and the big-Oh constant is $1$).  The idea is to borrow some of the growth from $r_1^n$ to show the main term in \eqref{eq:mainbigoh} is much larger than the secondary term, and thus almost all of the time the leading digits are determined by the main term. To do this, we choose a small $\epsilon$ (which is positive if $r_1 > 1$ and negative if $r_1 < 1$) and an $n_0$ such that
\ben \item $|r_2| < r_1^{1-\epsilon}$, and \item for all $n > n_0$,
$(L\gamma)^{1/n}/r_1^\epsilon < 1$. \een
%%%%%%%%, which then implies $L\gamma/r_1^{n\epsilon} = \left((L\gamma)^{1/n}/r_1^\epsilon\right)^n$. \een

As $L\gamma > 1$, $(L\gamma)^{1/n}$ is  decreasing to $1$ as $n$
tends to infinity. Note $\epsilon > 0$ if $r_1 > 1$ and $\epsilon < 0$
if $r_1 < 1$. Letting \be \beta \ := \
\frac{(L\gamma)^{1/n_0}}{r_1^\epsilon} \frac{|r_2|}{r_1^{1 - \epsilon}} \
< \ 1, \ee we find that the error term above is bounded by $\beta^n$ for $n
> n_0$, which tends to $0$.

We take logarithms, and will use $\log(1+x) = x + O(x^2)$ as $x\to 0$. Therefore
\bea y_n & \ = \ & \log_{10} a_n \nonumber\\ & = & \log_{10}(\gamma_1 r_1^n)
+ O\left( \log_{10} (1 + \beta^n) \right) \nonumber\\ & = & n \log_{10}
r_1 + \log_{10} \gamma_1 + O(\beta^n), \eea where the big-Oh constant is
bounded by $C$ say. As $\log_{10} r_1 \not\in \Q$, the fractional
parts of $n \log_{10} r_1$ are equidistributed modulo $1$, and hence
so are the shifts obtained by adding the fixed constant $\log_{10}
u_1$.

To complete the proof, we have to show the error term $O(\beta^n)$ is negligible. The problem is that
it can change the first digit; for
example, if we had $999999$ (or $1000000$), then if the error term
contributes $2$ (or $-2$), we would change the first digit. We thus need some weak control over how often this can happen. For $n$ sufficiently large, the error term will
change a vanishingly small number of first digits. To see this, suppose  $n \log_{10}
r_1 + \log_{10} \gamma_1$ exponentiates to first digit $j-1$. This means \be n \log_{10} r_1 + \log_{10} \gamma_1 \
\in \ [p_{j-1},p_j)., \ee where $p_k = \log_{10} k$. As the error term is at most
$C\beta^n$, $y_n$ exponentiates to a different first digit than
$n \log_{10} r_1 + \log_{10} u_1$ only if one of the following holds:

\ben \item $n \log_{10} r_1 + \log_{10} \gamma_1$ is within $C\beta^n$ of
$p_j$, and adding the error term pushes us to or past $p_j$; \item
$n \log_{10} r_1 + \log_{10} \gamma_1$ is within $C\beta^n$ of $p_{j-1}$,
and adding the error term pushes us before $p_{j-1}$. \een

The first set is contained in $[p_j - C\beta^n, p_j)$, of length
$C\beta^n$. The second is contained in $[p_{j-1}, p_{j-1} +
C\beta^n)$, also of length $C\beta^n$. Thus the length of the
interval where $n \log_{10} r_1 + \log_{10} \gamma_1$ and $y_n$ could
exponentiate to different first digits is of size at most $2C
\beta^n$. If we choose $N$ sufficiently large then for all $n > N$
we can make these lengths arbitrarily small. As $n \log_{10} r_1 +
\log_{10} \gamma_1$ is equidistributed modulo $1$, we can control the size
of the subsets of $[0,1)$ where $n \log_{10} r_1 + \log_{10} \gamma_1$ and
$y_n$ disagree. The Benford behavior of $a_n$ now follows
in the limit. \hfill $\Box$ \end{proof}

%%%%%%%%%%%%%%%%%%%%%%%%%%%%%%%%%%%%%%%%%%%%%%%%%%%%%%%%%%%%%%%%%%%%%%%%%%%%%%%%%%%%%%%%%%%%%%%%%%%%%%%%%%%%%%%%%%%%%%%%%%%%%%%%%%%%%%%%%%%%%%%%%%%%%%%%%%
%%%%%%%%%%%%%%%%%%%%%%%%%%%%%%%%%%%%%%%%%%%%%%%%%%%%%%%%%%%%%%%%%%%%%%%%%%%%%%%%%%%%%%%%%%%%%%%%%%%%%%%%%%%%%%%%%%%%%%%%%%%%%%%%%%%%%%%%%%%%%%%%%%%%%%%%%%
%%%%%%%%%%%%%%%%%%%%%%%%%%%%%%%%%%%%%%%%%%%%%%%%%%%%%%%%%%%%%%%%%%%%%%%%%%%%%%%%%%%%%%%%%%%%%%%%%%%%%%%%%%%%%%%%%%%%%%%%%%%%%%%%%%%%%%%%%%%%%%%%%%%%%%%%%%
\section{Linear Recurrence Relations with Non-constant Coefficients} \label{section:order2}

%%%%%%%%%%%%%%%%%%%%%%%%%%%%%%%%%%%%%%%%%%%%%%%%%%%
%%%%%%%%%%%%%%%%%%%%%%%%%%%%%%%%%%%%%%%%%%%%%%%%%%%
%%%%%%%%%%%%%%%%%%%%%%%%%%%%%%%%%%%%%%%%%%%%%%%%%%%
\subsection{Set-up}

Building on our successful analysis of recurrence relations with constant coefficients, we turn to our new results for recurrences with non-constant coefficients. We start with recurrences of the form
\begin{equation} \label{eq:recurWithfg}
a_{n+1} \ = \ f(n)a_n + g(n)a_{n-1}
\end{equation}
where $f$ and $g$ are fixed functions and $g$ is never zero, and we choose initial values $a_1$ and $a_2$. We explore conditions on $f$ and $g$ such that the sequence generated obeys Benford's Law for all non-zero initial values.

We begin by introducing auxiliary functions $\lambda$ and $\mu$ and reduce \eqref{eq:recurWithfg} into a new recurrence with lower degree. This idea is similar in spirit to the approaches to solve the cubic and quartic by looking at related polynomials with lower degree.\footnote{Unlike our case, for polynomials this process breaks down for degree five and higher.} The goal is to obtain a recurrence relation where $a_{n+1}$ only depends on $a_{n}$ and these new functions, as then we can immediately read off solutions. Suppose there were $\lambda(n),\mu(n)$ such that
\begin{equation} \label{eq:recurWithLamMu}
a_{n+1} - \lambda(n)a_n \ = \ \mu(n)(a_n - \lambda(n-1)a_{n-1})
\end{equation}
for $n \geq 2$. Now we can define an auxiliary sequence $\{b_n\}_{n=1}^\infty$ by
\begin{equation}
b_n \ = \ a_{n+1} - \lambda(n)a_n
\end{equation}
for $n \geq 1$. We get recurrence relations of degree $1$ for both $\{a_n\}$ and $\{b_n\}$:
\begin{equation} \label{eq:recuran}
a_{n+1} \ = \ \lambda(n)a_n + b_n
\end{equation}
and
\begin{equation} \label{eq:recurbn}
b_n \ = \ \mu(n)b_{n-1}.
\end{equation}
These recurrence relations with lower degree make the following computations much easier.

We simplify \eqref{eq:recurWithLamMu} to
\begin{equation}
a_{n+1} \ = \ (\lambda(n) + \mu(n))a_n - \mu(n)\lambda(n-1)a_{n-1}.
\end{equation}
Comparing coefficients with those of the original recurrence relation, we see that it suffices for $\lambda(n),\mu(n)$ to satisfy
\begin{eqnarray} \label{fgfromlambdamu}
	f(n)\ &\ =\ & \lambda(n)+\mu(n), \nonumber \\
    g(n)\ &=& -\lambda(n-1)\mu(n).
\end{eqnarray}

Therefore, if given $f(n)$ and $g(n)$ we can find such functions $\lambda$ and $\mu$, we obtain recurrence relations of degree 1. In Lemma \ref{solvinglambdamu} we prove that functions $\lambda$ and $\mu$ always exist for any given pair $f(n),g(n)$, and in Lemma \ref{bnonzero} we show that they may be chosen so the sequence $\{b_n\}$ is non-vanishing. We remark that the functions $\lambda$ and $\mu$ will not be unique; in fact there will be infinitely many, parametrized by a real number. We move these calculations to \S\ref{section:solveCoeff} so as not to interrupt the flow of the proof, as the constructions are straightforward.

We proceed to solve for the closed form of $\{a_n\}$ in terms of $\lambda$ and $\mu$. By \eqref{eq:recurbn}, we have
\begin{equation} \label{eq:closedbn}
b_n \ = \ \mu(n)b_{n-1} \ = \ \mu(n)\mu(n-1)b_{n-2} \ = \ \cdots \ = \ \left(\prod_{i = 2}^{n} \mu(i)\right) b_1.
\end{equation}

By \eqref{eq:recuran}, we get
\begin{equation} \label{eq:recurannew}
\frac{a_{n+1}}{\prod_{i=1}^{n} \lambda(i)} \ = \ \frac{a_{n}}{\prod_{i=1}^{n-1} \lambda(i)} + \frac{b_{n}}{\prod_{i=1}^{n} \lambda(i)}.
\end{equation}

In the previous recurrence, we replace $n$ with $1, 2, \dots, n-1$ and substitute the results into the RHS of equation \eqref{eq:recurannew}. This gives
\begin{equation}
\frac{a_{n+1}}{\prod_{i=1}^{n} \lambda(i)} \ = \ \frac{a_2}{\lambda(1)} + \sum_{i = 2}^{n} \frac{b_i}{\prod_{j = 1}^{i}\lambda(j)}.
\end{equation}

We then multiply through by the denominator $\prod_{i=1}^{n} \lambda(i)$ and substitute in the closed form for $b_i$ from \eqref{eq:closedbn}. This gives us the closed form of the sequence $\{a_n\}$,
\begin{equation} \label{eq:closedan}
a_{n+1} \ = \ b_1 \left(\sum_{k=2}^{n} \prod_{i=k+1}^{n}\lambda(i) \prod_{j=2}^{k}\mu(j)\right) + a_2\prod_{i=2}^{n}\lambda(i).
\end{equation}

To simplify notation we define \be r(n) \ :=\ b_1\prod_{i=2}^{n} \mu(i) \ \ \ {\rm and}\ \ \  p(n)\ :=\ \frac{\lambda(n)}{\mu(n)};\ee we know $\mu$ is non-vanishing as $g$ is never zero and $g(n) = -\lambda(n-1)\mu(n)$. Under this notation, we rewrite the closed form of $\{a_n\}$ as
\begin{align}
a_{n+1}\ &=\ r(n)\left(1 + \frac{\lambda(n)}{\mu(n)} + \frac{\lambda(n)\lambda(n-1)}{\mu(n)\mu(n-1)} + \cdots + \frac{a_2}{b_1}\frac{\lambda(n)\cdots\lambda(2)}{\mu(n)\cdots\mu(2)}\right)\nonumber \\
    \ &=\ r(n)\left(1 + \sum_{k=3}^{n} \prod_{i=k}^{n} p(i) + \frac{a_2}{b_1} \prod_{i=2}^{n} p(i) \right). \label{eq:closedanrn}
\end{align}

%%%%%%%%%%%%%%%%%%%%%%%%%%%%%%%%%%%%%%%%%%%%%%%%%%%
%%%%%%%%%%%%%%%%%%%%%%%%%%%%%%%%%%%%%%%%%%%%%%%%%%%
%%%%%%%%%%%%%%%%%%%%%%%%%%%%%%%%%%%%%%%%%%%%%%%%%%%
\subsection{Analysis of Main and Secondary Terms}

We now perform an asymptotic analysis and show, for suitable choices of $\mu$ and $\lambda$, that the main term dominates and the behavior is equivalent to equidistribution problems that are previously studied or tractable. We give further conditions on $p(n)$ such that $a_{n+1}$ is asymptotically equivalent to $r(n)$ as $n\to\infty$.

\begin{lem}\label{lem:whenrdominates}
Let $p(n)$ be a function from $\N$ to $\R$ such that $\lim_{n \to \infty} p(n) = 0$. Then $\lim_{n\to\infty} \left(a_{n+1} - r(n)\right) = 0$.
% Suppose $p(n)$ be a non-increasing function from $\N \to \R$ such that $\lim_{n \to \infty} p(n) = 0$. Then $\lim_{n\to\infty} \left(a_{n+1} - r(n)\right) = 0$.
\end{lem}

\begin{proof}
As previously computed, the closed form of $\{a_n\}$ is given by \eqref{eq:closedanrn}, so
\begin{align}
|a_{n+1} - r(n)| \ &\leq\ |r(n)| \left(\sum_{k=3}^{n} \prod_{i=k}^{n} |p(i)| + \frac{a_2}{b_1} \prod_{i=2}^{n} |p(i)| \right) \nonumber \\
				\ &\leq \ |r(n)| \left( \max \left(1,\frac{a_2}{b_1}\right) \sum_{k=2}^{n} \prod_{i=k}^{n} |p(i)| \right).
\end{align}

Therefore, to show that $r(n)$ is the dominating part of $a_{n+1}$, it suffices to show that
\begin{align} \label{eq:errTermLimit}
	\lim_{n\to\infty} \sum_{k=2}^{n} \prod_{i=k}^{n} |p(i)|\ = \ 0.
\end{align}

Without loss of generality, suppose $p(n)$ is positive for all $n$. Denote $q(n) := \sum_{k=2}^{n} \prod_{i=k}^{n} p(i)$. Then we have that $q(n+1) = p(n+1)(1+q(n))$ and that $q(n)>0$. Fix $\epsilon>0$. There exists $N$ such that for all $n>N$, $|p(n)|<\epsilon$. So
\begin{align}
q(N+1) \ &<\  p(N+1)(1+q(N)) \ <\ \epsilon + \epsilon q(N), \nonumber \\
q(N+2) \ &<\  p(N+2)(1+q(N+1)) \ <\ \epsilon + \epsilon^2 + \epsilon^2 q(N), \nonumber \\
\vdots \nonumber \\
q(N+k) \ &<\  p(N+k)(1+q(N+k-1)) \ <\ \epsilon + \epsilon^2 + \cdots + \epsilon^k + \epsilon^k q(N).
\end{align}
For any given $\epsilon$, $q(N)$ is also fixed. Taking the limit as $k \to \infty$, we get that
\begin{equation}
q(N+k) \ <\ \epsilon + \epsilon^2 + \cdots + \epsilon^k + \epsilon^k q(N) \ \to\ \frac{\epsilon}{1-\epsilon}.
\end{equation}
Then computing the limit as $\epsilon \to 0$, we see $q(N+k)$ converges to 0 as well, which implies \eqref{eq:errTermLimit}. \hfill $\Box$
\end{proof}

The previous lemma gives conditions on the functions $\lambda$ and $\mu$ for the sequence $\{a_n\}$ to be dominated by the main term, $r(n)$. We now want to give equivalent conditions on the functions $f$ and $g$, which appear in the original recurrence relation.

\begin{lem} \label{lem:mainTermCond}
Given functions $f(n)$ and $g(n)$ with $f$ non-decreasing and their resulting auxiliary functions $\lambda(n)$ and $\mu(n)$ as above, then $\lim_{n \to \infty} p(n) = 0$ is equivalent to \be\lim_{n \to \infty} \frac{g(n)}{f(n)^2}\ =\ 0.\ee
\end{lem}
\begin{proof}

%{\color{red} [Note: the questioned (?) places assume $\lim_{n\to \infty} \mu(n-1)/\mu(n) < \infty$ and $\lim_{n\to \infty}\lambda(n-1)/\lambda(n) < \infty$, respectively. Will add hypotheses later.]}
%Because $g$ is non-decreasing, we have for all $n$ that
%$$1\ge \frac{g(n-1)}{g(n)} = \frac{\lambda(n-2)}{\lambda(n-1)}\cdot \frac{\mu(n-1)}{\mu(n)}$$
Because $f$ is non-decreasing, we have that for all $n$,
\begin{equation}\label{fincreasing}
f(n) \ =\ \lambda(n)+\mu(n) \geq \ \lambda(n-1)+\mu(n-1) \ =\ f(n-1).
\end{equation}
So for each $n$, either $\lambda(n) \geq \lambda(n-1)$ or $\mu(n) \geq \mu(n-1)$ (or both).

First assume that $\lim_{n \to \infty} p(n) = 0$. Then since
\begin{align}
f(n) \ &=\ \lambda(n)+\mu(n) \nonumber \\
&= \ (1+p(n))\mu(n), \nonumber \\
g(n) \ &=\ -\lambda(n-1)\mu(n) \nonumber \\
&= \ -p(n-1)\mu(n-1)\mu(n), %\overset{(?)}{\approx} -p(n)\mu(n)^2.
\end{align}
we have
\begin{align}
\lim_{n \to \infty} \frac{g(n)}{f(n)^2}
\ &= \ \lim_{n \to \infty} \frac{-p(n-1)\mu(n-1)\mu(n)}{(1+p(n))^2\mu(n)^2} \nonumber\\
\ &= \ -\lim_{n \to \infty} p(n-1)\frac{\mu(n-1)}{\mu(n)} \nonumber \\
\ &= \ -\lim_{n \to \infty} p(n) \frac{\lambda(n-1)}{\lambda(n)},
\end{align}
and since at least one of $\mu(n-1)/\mu(n)$, $\lambda(n-1)/\lambda(n)$ is less than or equal to 1 for each $n$, the limit goes to zero.

On the other hand, suppose $\lim_{n \to \infty} g(n)/f(n)^2 = 0$. Under this assumption, it must be that $f(n)=\lambda(n)+\mu(n)$ is eventually nonzero. Then
\begin{align}
0 \ &=\ \lim_{n \to \infty} \frac{g(n)}{f(n)^2} \ = \ \lim_{n \to \infty} \frac{-\lambda(n-1)\mu(n)}{(\lambda(n)+\mu(n))^2} \nonumber \\
\ &= \ -\lim_{n \to \infty} \frac{p(n)}{(p(n)+1)^2}\cdot \frac{\lambda(n-1)}{\lambda(n)} \ =\ -\lim_{n \to \infty}\frac{p(n-1)}{(p(n)+1)^2}\cdot \frac{\mu(n-1)}{\mu(n)},
\end{align}
which, by the same reasoning, shows that $\lim_{n \to \infty} p(n)/(p(n)+1)^2=0$ and hence that
$\lim_{n \to \infty} p(n) = 0$ as desired. \hfill $\Box$
\end{proof}

\begin{thm}
Suppose functions $f(n)$ and $g(n)$ with $f$ non-decreasing satisfy \be\lim_{n \to \infty} g(n)/f(n)^2 = 0.\ee Then $\{a_n\}$ is Benford if and only if $r(n)$ is Benford.
\end{thm}
\begin{proof}
Since $\lim_{n \to \infty} g(n)/f(n)^2 = 0$ and $f$ is non-decreasing, by Lemma \ref{lem:mainTermCond} we have that $\lim_{n\to \infty}p(n) = 0$. Then by Lemma \ref{lem:whenrdominates}, $\lim_{n\to \infty}(a_{n+1}-r(n)) = 0$. Hence by Lemma \ref{lem:benfordAsym}, if $r(n)$ is Benford then $a_{n+1}$ is Benford, which is equivalent to $a_n$ being Benford; and similarly if $a_n$ is Benford, then $r(n)$ is Benford. \hfill $\Box$
\end{proof}

%%%%%%%%%%%%%%%%%%%%%%%%%%%%%%%%%%%%%%%%%%%%%%%%%%%
%%%%%%%%%%%%%%%%%%%%%%%%%%%%%%%%%%%%%%%%%%%%%%%%%%%
%%%%%%%%%%%%%%%%%%%%%%%%%%%%%%%%%%%%%%%%%%%%%%%%%%%
\subsection{Constructing and Parametrizing the Coefficient Functions}\label{section:solveCoeff}

%%\section{Solving for Coefficient Functions and Parameterizing Functions}\label{section:solveCoeff}

In this section, we provide a construction for the desired auxiliary functions $\lambda$ and $\mu$, and show that this can be done in a way that avoids vanishing denominators in the computations of the previous section.

\begin{lem} \label{solvinglambdamu}
Given functions $f,g:\N_{\ge 2} \to \R$, there exist\footnote{These functions are not uniquely determined by $f,g$; however, they are completely determined by a choice of $\lambda(1)$.} functions $\lambda:\N_{\ge 1} \to \R$ and $\mu : \N_{\ge 2} \to \R$ such that for all $n\ge 2$,
\begin{align} \label{fgfromlambdamueqs}
	f(n)\ &=\ \lambda(n)+\mu(n),\nonumber \\
    g(n)\ &=\ -\lambda(n-1)\mu(n).
\end{align}
\end{lem}

\begin{proof}
Let $\{\alpha_n\}_{n=1}^\infty,\{\beta_n\}_{n=1}^\infty$ be the sequences both satisfying the same recurrence relation
\begin{align}
\alpha_{n+1} \ &=\ f(n)\alpha_n + g(n)\alpha_{n-1}, \nonumber \\
\beta_{n+1} \ &=\ f(n)\beta_n + g(n)\beta_{n-1}
\end{align}
with initial terms
\begin{align}
\alpha_1 \ = \ 0,& \; \alpha_2 \ = \ 1, \nonumber\\
\beta_1 \ = \ 1,& \; \beta_2 \ = \ 0.
\end{align}

Choose $c \in \R \setminus\{-\beta_k/\alpha_k : k\in \N , \alpha_k\ne 0\}$ to be an arbitrary constant.
Now we define
\begin{align}
\lambda_c(n) \ &=\ \frac{\alpha_nc + \beta_n}{\alpha_{n-1}c + \beta_{n-1}}, \nonumber\\
\mu_c(n) \ &=\ f(n) - \lambda_c(n).
\end{align}

For any choice of $c$ as above, setting $\lambda = \lambda_c$ and $\mu = \mu_c$ satisfies \eqref{fgfromlambdamueqs}, by the following.
The equation \begin{equation}
f(n) \ = \ \lambda_c(n) + \mu_c(n)
\end{equation} follows from the way we've specified $\mu_c$.
Now, let $\lambda_c(1) = c$. Then
\begin{equation}
g(n) \ = \ -\lambda_c(n-1)\mu_c(n)
\end{equation} follows by induction on $n$; since we chose the constant $c \notin \{-\beta_k/\alpha_k : k\in \N, \alpha_k\ne 0\}$, the denominator of our expression for $\lambda_c$ never vanishes. \hfill $\Box$
\end{proof}

\begin{lem} \label{bnonzero}
Given a choice of initial terms $a_1,a_2$ for the recurrent sequence, we can find $\lambda,\mu$ as above so that the sequence $\{b_n\}_{n=1}^\infty$ defined by \begin{equation}b_n \ = \ a_{n+1} - \lambda (n) a_n \end{equation}
does not vanish.
\end{lem}

\begin{proof}
The sequence $\{b_n\}$ satisfies the recurrence
\begin{equation}b_n \ = \ \mu(n)b_{n-1} \end{equation}
for $n\ge 2$; moreover, $\mu$ is non-vanishing, since we required $g$ to be non-vanishing and \begin{equation} g(n) \ =\ -\lambda(n-1)\mu(n). \end{equation} Hence as long as $b_1 \neq 0$, the entire sequence $\{b_n\}$ is never zero.

Since \begin{equation}b_1 \ = \ a_2 - \lambda(1)a_1, \end{equation} we only need $\lambda(1) \neq a_2/a_1$. Since the constant $c=\lambda (1)$ was arbitrarily chosen in Lemma \ref{solvinglambdamu} from the real numbers minus a countable set, we can still choose $c$ to avoid one more number, so that $c \neq a_2 / a_1$. \hfill $\Box$
\end{proof}

%%%%%%%%%%%%%%%%%%%%%%%%%%%%%%%%%%%%%%%%%%%%%%%%%%%%%%%%%%%%%%%%%%%%%%%%%%%%%%%%%%%%%%%%%%%%%%%%%%%%%%%%%%%%%%%%%%%%%%%%%%%%%%%%%%%%%%%%%%%%%%%%%%%%%%%%%%
%%%%%%%%%%%%%%%%%%%%%%%%%%%%%%%%%%%%%%%%%%%%%%%%%%%%%%%%%%%%%%%%%%%%%%%%%%%%%%%%%%%%%%%%%%%%%%%%%%%%%%%%%%%%%%%%%%%%%%%%%%%%%%%%%%%%%%%%%%%%%%%%%%%%%%%%%%
%%%%%%%%%%%%%%%%%%%%%%%%%%%%%%%%%%%%%%%%%%%%%%%%%%%%%%%%%%%%%%%%%%%%%%%%%%%%%%%%%%%%%%%%%%%%%%%%%%%%%%%%%%%%%%%%%%%%%%%%%%%%%%%%%%%%%%%%%%%%%%%%%%%%%%%%%%
\section{Examples of Benford Behavior in Non-constant Recurrences} \label{section:examples}

We saw in previous sections that as long as the coefficient functions $f$ and $g$ satisfy certain conditions, the error term vanishes in the limit, and hence by Lemma \ref{lem:benfordAsym} the main term $r(n) = b_1 \prod_{i=2}^{n} \mu(i)$ dominates. Then by Lemma \ref{lem:benfordEquidistr}, Benfordness is equivalent to an equidistribution problem. Since Benfordness is preserved under translation and dilation, it suffices to study $\prod_{i=1}^{n} \mu(i)$. For simplicity, we redefine
\begin{equation}
r(n) \ =\ \prod_{i=1}^{n} \mu(i).
\end{equation}
In this section, we give several examples of $\mu(n)$ that make $\{a_n\}$ a Benford sequence.

\begin{exa} \label{exa:factorialPower}
If $\mu(k) = ak$ where $a \in \R$, then $r(n) = a^n n!$. In this case, $r(n)$ follows Benford's Law. In the special case where $a=1$, we get the factorial function.
\end{exa}

\begin{proof}
Theorem 3 of \cite{Di} shows that the factorial function is Benford; the proof uses the lemma that $f(n)=b(n+\frac{1}{2})\log{n}+cn$ is equidistributed mod 1 ($b,c$ are constants). Taking $b=1$ and $c=\log{a}$ gives that $\{(n+\frac{1}{2})\log{n}+n\log a\}$ is equidistributed mod 1. Since equidistribution is preserved under translation, it follows from Stirling's formula and Lemma \ref{lem:benfordAsym} that
\begin{equation}
\log r(n) \ =\ \left(n+\frac{1}{2}\right)\log{n}+n\log a+\frac{1}{2}\log{2\pi}
\end{equation}
is equidsitributed mod 1. Consequently, $r(n)$ is a Benford sequence. \hfill $\Box$
\end{proof}

% \begin{proof} It suffices to show that $\log{k^nn!}$ is equidistribued mod 1. Using Stirling's Formula we find that
% \begin{align*}
%     \log{k^nn!}&=n\log{k}+n\log{n}-n\log{e}+\frac{1}{2}\log{2\pi n}\\
%     &=(n+\frac{1}{2})\log{n}+(\log{k}-\log{e})n+\frac{1}{2}\log{2\pi}
% \end{align*}
% Now let $g(n):=(n+\frac{1}{2})\log{n}+(\log{k}-\log{e})n+\frac{1}{2}\log{2\pi}$ and $f(n):=a(n+\frac{1}{2})\log{n}+bn$. Then for $a=1$ and $b=\log{k}-\log{e}$ we have that
% \begin{align*}
%     \lim_{n\to\infty}f(n)-g(n)=-\frac{1}{2}\log{2\pi}
% \end{align*}
% From Lemma 1.6, which says that if $\{x_n\}$ is uniformly distributed mod 1 and there exists $\{y_n\}$ such that $\lim_{n\to\infty}(x_n-y_n)=\alpha$ with $\alpha\not\in\R$ then $\{y_n\}$ is uniformly distributed mod 1.
% Since $a\in\R$, it suffices to show that $f(n)$ is equidistributed mod 1. From the proof that $n!$ is Benford we have that $f(n)$ is equidistributed mod 1 for all constants $a,b$.
% \end{proof}

% \begin{exa} ({\color{red} ??})(Linear $\mu$?)
% If $\mu(k) = k + b$ where $b \in \R$, then $r(n) = ?$.
% \end{exa}

\begin{exa}
If $\mu(k) = k^\alpha$ where $\alpha \in \R$, then $r(n) = (n!)^\alpha$ follows Benford's Law.
\end{exa}

\begin{proof}
It suffices to show $\log r(n) = \alpha \log(n!)$ is equidistributed mod 1. By Example \ref{exa:factorialPower}, we know $\{\log(n!)\}$ is equidistributed mod 1. The result follows immediately since multiplying by a constant preserves equidistribution. \hfill $\Box$
\end{proof}

\begin{lem}[Weyl Equidistribution Theorem for Polynomials] \label{lem:edPoly}
Let $s \geq 1$ be an integer, and let $P(n) = \alpha_s n^s + \cdots + \alpha_0$ be a polynomial of degree $s$ with $\alpha_0,\ldots,\alpha_s \in \R$. If $\alpha_s$ is irrational, then $n \mapsto P(n)$ is asymptotically equidistributed on $\Z$.
\end{lem}

See \cite{Tao} for a proof of the above lemma, which we use to construct the following example.

\begin{exa}
If $\mu(k) = \exp(\alpha h(k))$ where $\alpha$ is irrational and $h(k)$ is a monic polynomial, then $r(n)$ follows Benford's law.
\end{exa}

\begin{proof}
Note that
$\log r(n) = \alpha \sum_{k=1}^{n} h(k)$. Since $h(k)$ is a monic polynomial, the sum $P(n) = \sum_{k=1}^{n} h(k)$ is a polynomial with rational leading coefficient. By Lemma \ref{lem:edPoly}, $\log r(n) = \alpha P(n)$ has irrational leading coefficient, so it is asymptotically equidistributed mod 1. Thus $r(n)$ is Benford. \hfill $\Box$
\end{proof}

In the above examples, $f(n)$ and $g(n)$ are all deterministic functions. However, we may not be able to find the exact forms of $f$ and $g$ in many cases, or as we saw in the introduction we may wish them to be random variables to model non-deterministic real world processes. Here we extend our results and give an example where $\mu$ is a random variable.

\begin{exa} \label{exa:rv}
Suppose $\mu(n) \sim h(n) U_n$, where $U_n$ are independent random variables uniformly distributed on $[0, 1]$ and $h(n)$ is a deterministic function in $n$ such that $\prod_{i=1}^{n} h(i)$ is Benford. Then $r(n) \sim \prod_{i=1}^{n} h(i) \prod_{i=1}^{n} U_i$, and $r(n)$ follows Benford's Law.
% {\color{red}(What does $\sim$ means? Should the conclusion be $r(n) \sim \prod_{i=1}^{n} h(i) \prod_{i=1}^{n} U_i$?)}
\end{exa}

\begin{rek}
As in Lemma 1.2 of \cite{JKKKM}, one can show that chains of uniform distributions converge to Benford's Law rapidly using the Mellin transform. A chain of uniform distributions multiplied together, as in $\prod_{i=1}^{n} U_i$, follows Benford's Law. The product of $h(n)$ will be Benford if $h$ is $\mu$ from any one of the above examples. Upon taking logarithms, both $\log (\prod_{i=1}^{n} U_i)$ and $\log (\prod_{i=1}^{n} h(i))$ will be equidistributed mod 1. Since the sum mod 1 of two independent equidistributed sequences is again equidistributed mod 1, in these cases $r(n)$ will be a Benford sequence.
\end{rek}

%%%%%%%%%%%%%%%%%%%%%%%%%%%%%%%%%%%%%%%%%%%%%%%%%%%%%%%%%%%%%%%%%%%%%%%%%%%%%%%%%%%%%%%%%%%%%%%%%%%%%%%%%%%%%%%%%%%%%%%%%%%%%%%%%%%%%%%%%%%%%%%%%%%%%%%%%%
%%%%%%%%%%%%%%%%%%%%%%%%%%%%%%%%%%%%%%%%%%%%%%%%%%%%%%%%%%%%%%%%%%%%%%%%%%%%%%%%%%%%%%%%%%%%%%%%%%%%%%%%%%%%%%%%%%%%%%%%%%%%%%%%%%%%%%%%%%%%%%%%%%%%%%%%%%
%%%%%%%%%%%%%%%%%%%%%%%%%%%%%%%%%%%%%%%%%%%%%%%%%%%%%%%%%%%%%%%%%%%%%%%%%%%%%%%%%%%%%%%%%%%%%%%%%%%%%%%%%%%%%%%%%%%%%%%%%%%%%%%%%%%%%%%%%%%%%%%%%%%%%%%%%%
\section{Generalization to Higher Depth Recurrence Relations} \label{section:higherDegree}

In this section, we generalize some of the results to linear recurrence relations of larger depth. A general recurrence relation of depth $L$ is of the form
\begin{equation} \label{eq:highRecur}
a_{n+1} = f_1(n)a_n + f_2(n)a_{n-1} + \cdots + f_{L}(n)a_{n-L+1}.
\end{equation}
As before, our goal is to find a closed form expression for the recurrent sequence, then to isolate a main term with simpler form which dominates. Our technique is again to reduce the degree of the recurrence by introducing auxiliary functions.

Here we demonstrate the process for computing an (asymptotic) closed form of a sequence satisfying a recurrence relation of degree 3. For recurrent sequences of higher degree, we can repeat similar processes until we reduce to the previously studied case of a recurrent sequence of degree 2.

Suppose the sequence $\{a_n\}_{n=1}^{\infty}$ is defined by the recurrence
\begin{align} \label{eq:highRecurAn}
a_{n+1} \ =\ f_1(n)a_n + f_2(n)a_{n-1} + f_3(n)a_{n-2}
\end{align}
for $n\ge 3$, and has initial values $a_1,a_2,a_3$. If we consider a linear combination of the adjacent terms of this sequence, the resulting sequence should satisfy a recurrence relation of degree 2. This suggests that we should define an auxiliary sequence $\{b_n\}_{n=1}^{\infty}$ by
\begin{align} \label{eq:highAnBn}
b_n \ =\ a_{n+1} - \lambda(n)a_n,
\end{align}
(with the function $\lambda(n)$ still to be determined) and posit the existence of functions $g_1,g_2$ so that it satisfies the recurrence relation
\begin{align} \label{eq:highRecurBn}
b_n \ =\ g_1(n-1)b_{n-1} + g_2(n-1)b_{n-2}.
\end{align}
We next substitute \eqref{eq:highAnBn} into \eqref{eq:highRecurBn} and compare coefficients with \eqref{eq:highRecurAn} to determine functions $\lambda$, $g_1$ and $g_2$:
\begin{multline}
a_{n+1} \ =\ (\lambda(n)+g_1(n-1))a_n
+ (-g_1(n-1)\lambda(n-1) \\
+ g_2(n-1))a_{n-1}
+ (-g_2(n-1)\lambda(n-2))a_{n-2}.
\end{multline}

We see that the previous relations hold if we can ensure that the following hold:
\begin{align}
f_1(n) \ &=\ \lambda(n) + g_1(n-1), \nonumber\\
f_2(n) \ &=\ -g_1(n-1)\lambda(n-1)+g_2(n-1), \nonumber\\
f_3(n) \ &=\ -g_2(n-1)\lambda(n-2). \label{eq:highFctnRelations}
\end{align}

We will show in Lemma \ref{lambdathree} that we can always find $\lambda,g_1,g_2$ which satisfy these relations.

Consider the recurrence relation \eqref{eq:highRecurBn} satisfied by $b_n$. By the results in previous sections, we can find auxiliary functions $\mu_1(n)$ and $\mu_2(n)$ so that
\begin{align}
g_1(n) \ &=\ \mu_1(n) + \mu_2(n), \nonumber \\
g_2(n) \ &=\ -\mu_1(n-1)\mu_2(n).
\end{align}

By Lemma~\ref{lem:mainTermCond} in Section~\ref{section:order2}, if $\mu_1(n)/\mu_2(n) \to 0$ as $n \to \infty$, then $\{b_n\}$ is dominated by the product $\prod_{i=1}^{n} \mu_2(i)$. By \eqref{eq:highAnBn}, we can solve for $a_n$:
\begin{equation} \label{eq:closedhighan}
a_{n+1} \ = \ b_1 \left( \sum_{k=2}^{n} \prod_{i=k+1}^{n}\lambda(i) \prod_{j=2}^{k-1}\mu_2(j) \right) + a_2\prod_{i=2}^{n}\lambda(i).
\end{equation}
As before, if $\lambda(n)/\mu_2(n) \to 0$, then $a_{n+1} - \prod_{i=1}^{n} \mu_2(i)\to 0$ as $n\to \infty$.

The relationship between the functions given in the recurrence relations and the auxiliary functions are given in \eqref{eq:highFctnRelations}. Thus, under the conditions that $\mu_1(n)/\mu_2(n) \to 0$ and $\lambda(n)/\mu_2(n) \to 0$, we have
\begin{align}
f_1(n) \ &=\ \mu_2(n) \left( 1+\frac{\lambda(n)}{\mu_2(n)} \right) \ \sim\  \mu_2(n), \nonumber \\
f_2(n) \ &=\ \mu_2(n)^2 \left( -\frac{\lambda(n)}{\mu_2(n)} - \frac{\mu_1(n)}{\mu_2(n)} \right), \nonumber \\
f_3(n) \ &=\ \mu_2(n)^3 \frac{\lambda(n)}{\mu_2(n)} \frac{\mu_1(n)}{\mu_2(n)},
\end{align}
and thus
\begin{align} \label{eq:highCond}
\frac{f_2(n)}{f_1(n)^2}\ \to\ 0 \quad \ \ \text{and}\ \ \quad \frac{f_3(n)}{f_1(n)^3}\ \to\ 0.
\end{align}

Conversely, we can show that, for suitable functions $f_1, f_2, f_3$, if \eqref{eq:highCond} holds, then $\{a_n\}$ is dominated by a multiplicative term. By Lemma \ref{lem:benfordAsym}, it suffices to consider Benfordness of the main term of $\{a_n\}$. Examples are given in Section \ref{section:examples}.
%%%%%%%%%%%%%%%%
%%%% The conditions!!
%%%%%%%%%%%%%%%%

The above is the case of degree 3. For even higher degree recurrences, similar results will hold, as long as a main term can be isolated; we can repeatedly introduce auxiliary sequences satisfying recurrence relations of lower degree, eventually reaching the degree 2 case we have studied. Along the way, we will need to impose conditions on the coefficient functions so that we can ignore lower degree terms.

\begin{lem}\label{lambdathree}
Given functions $f_1,f_2,f_3$, there exist\footnote{As in Lemma \ref{solvinglambdamu}, these functions are not unique, but are completely determined by $\lambda(1), \lambda(2)$.} functions $\lambda,g_1,g_2$ such that equations
\begin{align}
f_1(n) \ &=\ \lambda(n) + g_1(n-1), \nonumber \\
f_2(n) \ &=\ -g_1(n-1)\lambda(n-1)+g_2(n-1), \nonumber \\
f_3(n) \ &=\ -g_2(n-1)\lambda(n-2) \label{eq:solveforlambdaf1f2f3}
\end{align}
hold.
\end{lem}

\begin{proof}
Let $\{\alpha_n\}$, $\{\beta_n\}$, $\{\gamma_n\}$ be sequences satisfying the same recurrence relation
\begin{align}
\alpha_n \ &=\ f_1(n)\alpha_{n-1}+f_2(n)\alpha_{n-2}+f_3(n)\alpha_{n-3}, \nonumber \\
\beta_n \ &=\ f_1(n)\beta_{n-1}+f_2(n)\beta_{n-2}+f_3(n)\beta_{n-3}, \nonumber \\
\gamma_n \ &=\ f_1(n)\gamma_{n-1}+f_2(n)\gamma_{n-2}+f_3(n)\gamma_{n-3}
\end{align}
with initial terms
\begin{align}
\alpha_0 \ = \ 0, \; \alpha_1 \ = \ 0,& \; \alpha_2 \ = \ 1, \nonumber \\
\beta_0 \ = \ 0 , \; \beta_1 \ = \ 1,& \; \beta_2 \ = \ 0, \nonumber \\
\gamma_0 \ = \ 1, \; \gamma_1 \ = \ 0,& \; \gamma_2 \ = \ 0.
\end{align}
Choose $(c,d) \in \R^2\setminus \{(r,-(\alpha_k r + \gamma_k)/\beta_k) \ : \ r\in \R, \beta(k)\ne 0\}$ to be an arbitrary constant. Now we define
\begin{align}
\lambda_{c,d}(n) \ &= \ \frac{\alpha_n c + \beta_n d + \gamma_n}{\alpha_{n-1}c+\beta_{n-1}d + \gamma_{n-1}}, \nonumber \\
g_{1,c,d}(n) \ &= \ f_1(n+1) - \lambda_{c,d}(n+1), \nonumber \\
g_{2,c,d}(n) \ &= \ f_2(n+1) + g_{1,c,d}(n)\lambda_{c,d}(n).
\end{align}
The definitions of $g_{1,c,d},g_{2,c,d}$ ensure that first two equations of \eqref{eq:solveforlambdaf1f2f3} both hold; the third equation follows by induction. Since we chose $(c,d) \not\in \{(r,-(\alpha_k r + \gamma_k)/\beta_k) \ : \ r\in \R, \beta_k\ne 0\}$, the denominator of our expression for $\lambda_{c,d}$ never vanishes. \hfill $\Box$
\end{proof}

%%%%%%%%%%%%%%%%%%%%%%%%%%%%%%%%%%%%%%%%%%%%%%%%%%%%%%%%%%%%%%%%%%%%%%%%%%%%%%%%%%%%%%%%%%%%%%%%%%%%%%%%%%%%%%%%%%%%%%%%%%%%%%%%%%%%%%%%%%%%%%%%%%%%%%%%%%
%%%%%%%%%%%%%%%%%%%%%%%%%%%%%%%%%%%%%%%%%%%%%%%%%%%%%%%%%%%%%%%%%%%%%%%%%%%%%%%%%%%%%%%%%%%%%%%%%%%%%%%%%%%%%%%%%%%%%%%%%%%%%%%%%%%%%%%%%%%%%%%%%%%%%%%%%%
%%%%%%%%%%%%%%%%%%%%%%%%%%%%%%%%%%%%%%%%%%%%%%%%%%%%%%%%%%%%%%%%%%%%%%%%%%%%%%%%%%%%%%%%%%%%%%%%%%%%%%%%%%%%%%%%%%%%%%%%%%%%%%%%%%%%%%%%%%%%%%%%%%%%%%%%%%
\section{Generalization to Multiplicative Recurrence Relations} \label{section:multi}

So far, we have found that a large family of sequences defined by linear recurrences follow Benford's law. Inspired by an idea\footnote{Generalizing the recurrence relation of Fibonacci $k$-step numbers, the authors proved the recursive sequence $x_n = \prod_{i=1}^{k} x_{n-i}$ to be Benford. Each term $x_n$ ($n>k$) is the product of powers of $x_1, \cdots, x_k$ with Fibonacci $k$-step numbers as the exponents.} in \cite{RM}, we consider sequences generated by multiplicative recurrence relations. We find that we can use our results from the previous section to give conditions under which the sequence obeys Benford's law.

Suppose $f_1, f_2, \ldots, f_k$ are functions on $\N_{\geq 0}$ and define the sequence $\{A_n\}_{n=1}^{\infty}$ by the multiplicative recurrence relation
\begin{equation} \label{eq:multRecur}
A_{n+1} \ =\ A_n^{f_1(n)} A_{n-1}^{f_2(n)} \dots A_{n-k+1}^{f_k(n)}
\end{equation}
with initial values $A_1, \dots, A_k$. We see that $A_n$ is a product of powers of the initial terms $A_1, \dots, A_k$; the exponents satisfy the recurrence \eqref{eq:highRecur}.

In this section, we use $k=2$ as an example to illustrate the relationship between \eqref{eq:multRecur} and \eqref{eq:highRecur}, and then give a simple example of a sequence $\{A_n\}$ that obeys Benford's law.

Define sequence $\{A_n\}_{n=1}^{\infty}$ by the recurrence relation \begin{equation}
A_{n+1} \ =\ A_n^{f(n)} A_{n-1}^{g(n)}
\end{equation} with initial values $A_1$, $A_2$. Then we can express the elements of the sequence in terms of $A_1$ and $A_2$.

\begin{lem}
Let the sequence $\{A_n\}_{n=1}^{\infty}$ be as given above. Then the closed form of $A_n$ is $A_n = A_2^{x_n} A_1^{y_n}$, where the exponents $\{x_n\}$ and $\{y_n\}$ satisfy the linear recurrence relations
\begin{align}
x_{n+1} \ = \ f(n)x_n + g(n)x_{n-1}, \nonumber \\
y_{n+1} \ = \ f(n)y_n + g(n)y_{n-1}
\end{align}
with initial values
\begin{align}
x_1 \ = \ 0, x_2 \ = \ 1, \nonumber \\
y_1 \ = \ 1, y_2 \ = \ 0.
\end{align}
\end{lem}

\begin{proof}
Here we proceed by induction. The base cases are easily established by substituting in the initial values of the sequences $\{x_n\}$ and $\{y_n\}$. Now assume the recurrences hold for $\N_{\leq n}$. As
\begin{align}
A_{n+1} \ &=\ A_n^{f(n)}A_{n-1}^{g(n)} \nonumber  \\
        \ &=\ \left(A_2^{x_n}A_1^{y_n} \right)^{f(n)} \left(A_2^{x_{n-1}} A_1^{y_{n-1}} \right)^{g(n)} \nonumber  \\
        \ &=\ A_2^{f(n)x_n+g(n)x_{n-1}} A_1^{f(n)y_n+g(n)y_{n-1}},
\end{align}
we get
\begin{eqnarray}
x_{n+1} & \ = \ & f(n)x_n + g(n)x_{n-1}, \nonumber \\
y_{n+1} & \ = \ & f(n)y_n + g(n)y_{n-1}.
\end{eqnarray}
\hfill $\Box$
\end{proof}

Since both sequences $\{x_n\}$ and $\{y_n\}$ are generated by the same recurrence as \eqref{eq:recurWithfg}, by Section \ref{section:solveCoeff} we can find auxilary functions $\lambda(n)$ and $\mu(n)$. If the functions $f(n)$ and $g(n)$ satisfy the hypotheses of Lemma \ref{lem:mainTermCond}, then both $\{x_n\}$ and $\{y_n\}$ have asymptotic forms
\begin{eqnarray}
x_{n+1}  &\ \to\  &(x_2 - \lambda(1)x_1) \prod_{i=2}^{n} \mu(i)
        \ = \ \prod_{i=2}^{n} \mu(i), \nonumber\\
y_{n+1}  &\to& (y_2 - \lambda(1)y_1) \prod_{i=2}^{n} \mu(i)
        \ = \ -\lambda(1) \prod_{i=2}^{n} \mu(i)
\end{eqnarray}
as $n \to \infty$.

Therefore, under this condition,
\begin{align}
\log(A_{n+1}) \ &\to\ \left(\prod_{i=2}^{n}\mu(i)\right)\log(A_2) + \left(-\lambda(1) \prod_{i=2}^{n}\mu(i)\right) \log(A_1) \nonumber \\
			  \ &=\ \left(\prod_{i=2}^{n}\mu(i)\right) \left(\log(A_2)-\lambda(1)\log(A_1) \right).
\end{align}

By Lemmas \ref{lem:benfordEquidistr} and \ref{lem:benfordAsym}, if $\{c_0 \prod_{i=2}^{n}\mu(i)\}_{n=2}^{\infty}$ is equidistributed mod 1 where $c_0 = \log(A_2)-\lambda(1)\log(A_1)$ is a constant, then $\{A_n\}$ is a Benford sequence.

\begin{exa}
Let $\log(A_2)-\lambda(1)\log(A_1) = c_0 \notin \Q$ and $\mu(i) = \frac{P(i)}{P(i-1)}$ where $P(n)$ is a non-vanishing monic polynomial.
\end{exa}

This construction is immediately suggested by Lemma \ref{lem:edPoly}.

%%%%%%%%%%%%%%%%%%%%%%%%%%%%%%%%%%%%%%%%%%%%%%%%%%%%%%%%%%%%%%%%%%%%%%%%%%%%%%%%%%%%%%%%%%%%%%%%%%%%%%%%%%%%%%%%%%%%%%%%%%%%%%%%%%%%%%%%%%%%%%%%%%%%%%%%%%
%%%%%%%%%%%%%%%%%%%%%%%%%%%%%%%%%%%%%%%%%%%%%%%%%%%%%%%%%%%%%%%%%%%%%%%%%%%%%%%%%%%%%%%%%%%%%%%%%%%%%%%%%%%%%%%%%%%%%%%%%%%%%%%%%%%%%%%%%%%%%%%%%%%%%%%%%%
%%%%%%%%%%%%%%%%%%%%%%%%%%%%%%%%%%%%%%%%%%%%%%%%%%%%%%%%%%%%%%%%%%%%%%%%%%%%%%%%%%%%%%%%%%%%%%%%%%%%%%%%%%%%%%%%%%%%%%%%%%%%%%%%%%%%%%%%%%%%%%%%%%%%%%%%%%
\section{Questions and Future Research}

\begin{enumerate}
\item In Section \ref{section:order2}, we mainly consider the case when $\lambda(n)/\mu(n) \to 0$ as $n \to \infty$. The case $\lambda(n)/\mu(n) = r$ where $r \in \R$ for all but finitely many $n$ can also be analyzed in a similar fashion, and $\{a_n\}$ is again dominated by a multiplicative term. The challenge is when $\lambda(n)/\mu(n) \to \infty$. In this case, there is no single main term (at least the way we have chosen to reduce  our sequences).

\item In Example \ref{exa:rv}, we explore the case where $f$ and $g$ are random variables. This differs from the case where they are explicit functions, in that we have allowed randomness. Many processes, such as the example in the introduction from mathematical biology, can be described as recurrences with random variables as coefficients and thus it would be worthwhile exploring the Benfordness of such systems.

    %%In addition, conditions on $f$ and $g$ could be different, and we would like to know what conditions we should expect them to follow in these applications, as this would help %%decide which directions we should go in generalizing our results.

\end{enumerate}

\appendix

%%%%%%%%%%%%%%%%%%%%%%%%%%%%%%%%%%%%%%%%%%%%%%%%%%%%%%%%%%%%%%%%%%%%%%%%%%%%%%%%%%%%%%%%%%%%%%%%%%%%%%%%%%%%%%%%%%%%%%%%%%%%%%%%%%%%%%%%%%%%%%%%%%%%%%%%%%
%%%%%%%%%%%%%%%%%%%%%%%%%%%%%%%%%%%%%%%%%%%%%%%%%%%%%%%%%%%%%%%%%%%%%%%%%%%%%%%%%%%%%%%%%%%%%%%%%%%%%%%%%%%%%%%%%%%%%%%%%%%%%%%%%%%%%%%%%%%%%%%%%%%%%%%%%%
%%%%%%%%%%%%%%%%%%%%%%%%%%%%%%%%%%%%%%%%%%%%%%%%%%%%%%%%%%%%%%%%%%%%%%%%%%%%%%%%%%%%%%%%%%%%%%%%%%%%%%%%%%%%%%%%%%%%%%%%%%%%%%%%%%%%%%%%%%%%%%%%%%%%%%%%%%
\section{Proof of Weyl's Theorem}\label{sec:proofweyl}

Because of Lemma \ref{lem:benfordEquidistr}, to prove $\{x_n\}$ is Benford it suffices to prove the corresponding sequence $y_n = \log_{10} x_n \bmod 1$ is equidistributed. If $y_n = n \alpha \bmod 1$ (so $x_n$ is the geometric series $(e^{\alpha})^n$), Weyl's theorem yields it is equidistributed if and only if $\alpha$ is irrational. Before doing this, we first prove an easier result, Kronecker's Theorem, which states that if $\alpha$ is irrational than the sequence $n\alpha \bmod 1$ is dense in $[0,1)$; in other words, given any $x \in [0,1)$ and any $\epsilon > 0$ there exist an $n$ such that $|x - (n \alpha \bmod 1)| < \epsilon$. We give these proofs as they highlight the applicability of Fourier Analysis to understanding instances of Benford's law.

%%%%%%%%%%%%%%%%%%%%%%%%%%%%%%%%%%%%%%%%%%%%%%%%%%%
%%%%%%%%%%%%%%%%%%%%%%%%%%%%%%%%%%%%%%%%%%%%%%%%%%%
%%%%%%%%%%%%%%%%%%%%%%%%%%%%%%%%%%%%%%%%%%%%%%%%%%%
\subsection{Kronecker's Theorem}

\begin{thm}[Kronecker's Theorem] Let $\alpha$ be an irrational number. Then $\{n \alpha \bmod 1\}$ is dense in $[0,1)$. \end{thm}

\begin{proof} The prove is an immediate consequence of the Pigeonhole Principle: if we place $Q+1$ objects in $Q$ boxes, at least one box has two items.

Given a small $\epsilon > 0$, choose $Q$ so large that $1/Q < \epsilon$, and divide $[0,1)$ into $Q$ intervals: \be [0,1) \ = \ [0, 1/Q) \ \cup \ [1/Q, 2/Q) \ \cup \ \cdots \ \cup \ [(Q-1)/Q, 1). \ee If we look at the $Q+1$ numbers \be \alpha \bmod 1, \ \ \ 2 \alpha \bmod 1, \ \ \ \dots, \ \ \ Q\alpha \bmod 1, \ \ \ (Q+1)\alpha \bmod 1 \ee then at least two of them, say $m_1 \alpha \bmod 1$ and $m_2 \alpha \bmod 1$, must be in the same sub-interval of length $Q$. Thus as $m_i \alpha \bmod 1 = m_i \alpha - n_i$ for some integer $n_i$, we have \be \left|(m_1 \alpha - n_1) - (m_2 \alpha - n_2)\right| \ \le \ \frac1{Q} \ \ \ {\rm or} \ \ \ \left|(m_1-m_2)\alpha - (n_1-n_2)\right| \ \le \ \frac1{Q}. \ee

Letting $m = m_1-m_2$, we see that $|m\alpha \bmod 1|$ is at most $1/Q$; for convenience we assume it is positive and lives between $0$ and $1/Q$, though a similar argument would hold if it was just a little less than 1. Thus if we look at \be m\alpha \bmod 1, \ \ \ 2m\alpha \bmod 1, \ \ \ 3m\alpha \bmod 1, \ \ \ \dots \ee then each term is at most $1/Q$ from the previous in the above sequence, and we are always moving in the same direction (to the right), given any $x \in [0,1)$ we will eventually be within $1/Q$ of it. The reason is each term is at most $1/Q$ from the previous, so we cannot jump from more than $1/Q$ smaller than $x$ to more than $1/Q$ above. \hfill $\Box$\end{proof}

We thus see that it is relatively straightforward to establish the denseness of $n\alpha \bmod 1$ for irrational $\alpha$. Unfortunately for Benfordness we need it to be equidistributed, which is significantly more work.

%%%%%%%%%%%%%%%%%%%%%%%%%%%%%%%%%%%%%%%%%%%%%%%%%%%
%%%%%%%%%%%%%%%%%%%%%%%%%%%%%%%%%%%%%%%%%%%%%%%%%%%
%%%%%%%%%%%%%%%%%%%%%%%%%%%%%%%%%%%%%%%%%%%%%%%%%%%
\subsection{Weyl's Theorem}

We sketch the proof of Weyl's result. Actually, far more is true than stated in Theorem \ref{thm:weylequidistributionkis1}.

\begin{thm}[Weyl's Theorem]\label{thmweylequi} Let $\alpha$ be an irrational number
in $[0,1)$, and let $k$ be a fixed positive integer. Let $x_n =
\{n^k \alpha\}$. Then $\{x_n\}_{n=1}^\infty$ is equidistributed.
\end{thm}

\begin{proof} We expand slightly the proof from \cite{MT}, concentrating on the $k=1$ case. We start with some convenient notation. Define \be \twocase{\chi_{(a,b)}(x) \ :=
\ }{1}{if $x \in [a,b)$}{0}{otherwise;} \ee $\chi_{(a,b)}$ the characteristic (or indicator) function of the interval $[a,b)$.

We must show for any $[a,b) \subset [0, 1)$ that \be\label{eq:limchiab} \lim_{N \rightarrow
\infty} \frac{1}{N} \sum_{n=1}^N \chi_{(a,b)}(x_n) \ = \
b-a.\ee  The idea is to use a sequence of approximations common in analysis. We show that instead of proving a result for a sharp characteristic function it suffices to prove it for a continuous function. Then we show it is enough to prove it for a trigonometric polynomial. Such functions involve exponentials, as \be e_n(x) \ := \ e^{2\pi i n x}\ =\ \cos(2\pi n x) + i \sin(2\pi n x).\ee We will need to show that in the limit the resulting sums vanish if and only if $n \neq 0$. The reason we can do so is that our sums are geometric series, and through the geometric series formula we obtain closed form expressions.

Turning to the details, we first record a sum which will be useful later: \be
\begin{aligned} \frac{1}{N} \sum_{n = 1}^N e_m(x_n) &\ = \
\frac{1}{N} \sum_{n=1}^N e_m(n \alpha)
\\ &\ =\ \frac{1}{N} \sum_{n = 1}^N (e^{2\pi i m \alpha})^n \\
&\ = \  \begin{cases} 1 & \text{if $m=0$} \\ \frac{1}{N}
\frac{e_m(\alpha) - e_m((N+1) \alpha)}{1 - e_m(\alpha)} &
\text{if $m>0$,}  \end{cases}\end{aligned}  \ee where the last
follows from the geometric series formula. We now use the irrationality of $\alpha$, which implies that  $|1 - e_m(\alpha)|
> 0$. Thus for $m \neq
0$, \be\label{eq:fixedmlimNtoinfy} \lim_{N\to \infty} \frac{1}{N} \frac{e_m(\alpha) - e_m((N+1) \alpha)}{1 - e_m(\alpha)} \ = \
0.\ee

We now show that we can replace $\chi_{a,b}$ in \eqref{eq:limchiab} with an error as small as we desire.
Let \be P(x)\ =\ \sum_{m=-M}^M a_m e_m(x) \ee be a finite \emph{fixed} trigonometric polynomial; we we have chosen the range of indices to be symmetric, some of the coefficients may be zero and present just to simplify notation. It is essential that $M$ if fixed, independent of $N$, and a straightforward calculation shows\be \int_0^1 P(x)dx\ =\ a_0\ee (this is because sines and cosines integrate to zero over an integer number of periods, and the $n=0$ term is just the constant function 1).

Equation \eqref{eq:fixedmlimNtoinfy} implies that for any
\emph{fixed} finite trigonometric polynomial $P(x)$, we have \be
\lim_{N\to \infty} \frac{1}{N} \sum_{n=1}^N P(x_n) \ = \ a_0
\ = \ \int_0^1 P(x)dx; \ee it is essential here that $M$ is fixed, so we can take the limit as $N\to\infty$.

As $M$ is fixed, we may bound $|1 -
e_m(\alpha)|$ away from zero for all $m \in \{-M,\dots,M\}$. If $M$
varied with $N$, then $1-e_m(\alpha)$ could tend to $0$ for some
values of $m$ (depending on $N$). However, since $m$ is fixed  \bea \frac1{N} \sum_{n=1}^N P(x_n) & \ = \ & \frac1{N}
\sum_{n=1}^N \sum_{m=-M}^M a_m e_m(x_n) \nonumber\\ & = &
\sum_{m=-M}^M a_m \frac{1}{N} \sum_{n=1}^N e_m(x_n). \eea
Letting $A = \max_m |a_m|$ and $B = \min_m |1-e_m(\alpha)| > 0$,
from \eqref{eq:fixedmlimNtoinfy} we obtain \be \left|a_0 -
\frac1{N} \sum_{n=1}^N P(x_n)\right| \ \le \ \sum_{m=-M \atop m
\neq 0}^M A \cdot \frac{2}{NB} \ \le \ \frac{2MA}{B}
\frac1{N}, \ee which tends to zero as $N\to\infty$.

\begin{figure}[h]
\begin{center}
\scalebox{1}{\includegraphics{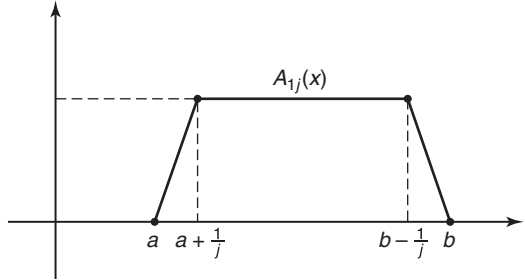}}\caption{\label{fig:weyl}Plot
of $A_{1j}(x)$. Note it is at most $\chi_{a,b}(x)$, and the two functions are equal in $[a+1/j, b-1/j]$. }
\end{center}\end{figure}

Consider two continuous approximations to the characteristic
function $\chi_{(a,b)}$ (see Figure \ref{fig:weyl} for a plot of one of them): \ben
\item $A_{1j}$: $A_{1j}(x) = 1$ if $a + \frac{1}{j} \le x \le b -
\frac{1}{j}$, drops linearly to $0$ at $a$ and $b$, and is zero
elsewhere (see Figure \ref{fig:weyl}).
\item $A_{2j}$: $A_{1j}(x) = 1$ if $a \le x \le b$, drops linearly
to $0$ at $a-\frac{1}{j}$ and $b+\frac{1}{j}$, and is zero
elsewhere.
\een

Note there are trivial modifications if $a = 0$ or $b = 1$.
Clearly \be A_{1j}(x) \ \le \ \chi_{(a,b)}(x) \ \le \ A_{2j}(x).
\ee Therefore \be \frac{1}{N} \sum_{n=1}^N A_{1j}(x_n) \ \le \
\frac{1}{N} \sum_{n=1}^N \chi_{(a,b)}(x_n) \ \le \
\frac{1}{N} \sum_{n=1}^N A_{2j}(x_n). \ee

We now need a result from Fourier Analysis, Fej\'{e}r's Theorem, on the weighted Fourier series associated to a continuous, periodic function $f$. For such an $F$, the Fourier coefficients are \be a_n \ := \ a_n(f) \ = \ \int_0^1 f(x) e^{-2\pi i n x} dx,\ee and the $N$\textsuperscript{th} Fej\'{e}r series $T_Nf$ is the weighted sum of the Fourier coefficients: \be T_Nf(x) \ := \ \sum_{n=-N}^N \left(1 - \frac{|n|}{N}\right) a_n e^{2\pi i n x}. \ee

\begin{thm}[Fej\'{e}r's Theorem]\label{thmfejer} Let $f(x)$ be a continuous, periodic
function on $[0,1]$. Given $\epsilon>0$ there exists an $N_0$ such
that for all $N > N_0$, \be |f(x) - T_N(x) |\ \leq\ \epsilon\ee
for every $x\in [0,1)$. Hence as $N\to\infty$, $T_Nf(x) \to f(x)$.
\end{thm}

The presence of the weight factors $(1 - |n|/N)$ gives the Fej\'{e}r series better convergence properties than the normal Fourier series. In particular, note that the coefficients close to $\pm N$ have a weight of almost zero, while those whose index is close to $0$ are almost unchanged.

By Fej\'{e}'s Theorem, for each $j$, given $\epsilon > 0$ we can find
symmetric trigonometric polynomials $P_{1j}(x)$ and $P_{2j}(x)$
such that $|P_{1j}(x) - A_{1j}(x)| < \epsilon$ and $|P_{2j}(x) -
A_{2j}(x)| < \epsilon$. As $A_{1j}$ and $A_{2j}$ are continuous
functions, we can replace \be \frac{1}{N} \sum_{n=1}^N
A_{ij}(x_n) \ \ \ \ \text{with} \ \ \ \ \frac{1}{N}
\sum_{n=1}^N P_{ij}(x_n) \ee at a cost of at most $\epsilon$. As
$N \rightarrow \infty$, \be \frac{1}{N} \sum_{n=1}^N
P_{ij}(x_n) \ \longrightarrow \ \int_0^1 P_{ij}(x)dx. \ee But
$\int_0^1 P_{1j}(x)dx = (b-a) - \frac{1}{j}$ and $\int_0^1
P_{2j}(x)dx = (b-a) + \frac{1}{j}$. Therefore, given $j$ and
$\epsilon$, we can choose $N$ large enough so that \be (b-a) -
\frac{1}{j} - \epsilon \ \le \ \frac{1}{N} \sum_{n=1}^N
\chi_{(a,b)}(x_n) \ \le \ (b-a) + \frac{1}{j} + \epsilon. \ee
Letting $j$ tend to $\infty$ and $\epsilon$ tend to $0$, we see
$\frac{1}{N} \sum_{n=1}^N \chi_{(a,b)}(x_n) \rightarrow b -
a$, completing the proof. \hfill $\Box$
\end{proof}

%%%%%%%%%%%%%%%%%%%%%%%%%%%%%%%%%%%%%%%%%%%%%%%%%%%%%%%%%%%%%%%%%%%%%%%%%%%%%%%%%%%%%%%%%%%%%%%%%%%%%%%%%%%%%%%%%%%%%%%%%%%%%%%%%%%%%%%%%%%%%%%%%%%%%%%%%%
%%%%%%%%%%%%%%%%%%%%%%%%%%%%%%%%%%%%%%%%%%%%%%%%%%%%%%%%%%%%%%%%%%%%%%%%%%%%%%%%%%%%%%%%%%%%%%%%%%%%%%%%%%%%%%%%%%%%%%%%%%%%%%%%%%%%%%%%%%%%%%%%%%%%%%%%%%
%%%%%%%%%%%%%%%%%%%%%%%%%%%%%%%%%%%%%%%%%%%%%%%%%%%%%%%%%%%%%%%%%%%%%%%%%%%%%%%%%%%%%%%%%%%%%%%%%%%%%%%%%%%%%%%%%%%%%%%%%%%%%%%%%%%%%%%%%%%%%%%%%%%%%%%%%%

\bigskip

\end{document}